\documentclass[11pt]{amsart}
\usepackage[english]{babel}
\usepackage{amsmath,amsthm}
\usepackage{amsfonts}
\usepackage{enumerate}
\usepackage{graphicx}
\usepackage{color}
\usepackage[all]{xy}
\usepackage{thmtools, thm-restate}
\usepackage{overpic}

\usepackage[top=1.5in, bottom=1.5in, left=1.5in, right=1.5in]{geometry}

\newtheorem{thm}{Theorem}
\newtheorem{cor}[thm]{Corollary}
\newtheorem{lem}[thm]{Lemma}

\newtheorem{conj}[thm]{Conjecture}

\newtheorem{exam}[thm]{Example}

\theoremstyle{definition}

\theoremstyle{remark}
\newtheorem{rem}[thm]{Remark}

\numberwithin{equation}{section}

\makeatletter
\newlength{\vbraceheight}
\def\vbig#1{{\resizebox{!}{\vbraceheight}{$\left#1\vbox to\vbraceheight{}\right.\n@space$}}}
\def\vbigl{\mathopen\vbig}
\def\vbigr{\mathclose\vbig}
\makeatother

\newcommand{\spin}{\ifmmode{\rm Spin}\else{${\rm spin}$\ }\fi}
\newcommand{\spinc}{\ifmmode{{\rm Spin}^c}\else{${\rm spin}^c$}\fi}

\newcommand{\Z}{\mathbb{Z}}

\newcommand{\Q}{\mathbb{Q}}

\newcommand{\M}{\mathcal{M}}

\begin{document}

\title{The Montesinos trick for proper rational tangle replacement}%

\author{Duncan McCoy}%
\address {Universit\'{e} du Qu\'{e}bec \`{a} Montr\'{e}al}
\email{mc\_coy.duncan@uqam.ca}

\author{Raphael Zentner}%
\address {Universität Regensburg}
\email{Raphael.Zentner@mathematik.uni-regensburg.de}

\date{}%

\begin{abstract}
Recently Iltgen, Lewark and Marino introduced the concept of a proper rational tangle replacement and the corresponding notion of the proper rational unknotting number. In this note we derive a version of the Montesinos trick for proper rational tangle replacement and use it to study knots with proper rational unknotting number one. We prove that knots with proper rational unknotting number one are prime and classify the alternating knots with proper rational unknotting number one. We also study Montesinos knots with proper rational unknotting number one.
\end{abstract}
\maketitle


\section{Introduction}
Recently Iltgen, Lewark and Marino introduced the concept of the proper rational unknotting number of a knot and derived a lower bound for this quantity from Khovanov homology \cite{Iltgen2021Khovanov}. A link $L'$ is said to be obtained by a {\em rational tangle replacement (RTR)} on $L$ if $L'$ can be obtained by taking a rational tangle $T$ in $L$ and replacing it by another rational tangle $T'$. This tangle replacement is said to be {\em proper} if the arcs of $T$ and $T'$ both connect the same pairs of endpoints on the boundary. So, for example, a crossing change is an example of a proper RTR, whereas resolving a crossing is not a proper RTR. Given a knot $K$ the {\em proper rational unknotting number} $u_q(K)$ is defined to be the minimal number of proper RTRs in a sequence of proper RTRs converting $K$ to the unknot. Since a crossing change is an example of a proper RTR, we see that the classical unknotting number $u(K)$ is an upper bound for $u_q(K)$. Another example of a proper RTR is shown in Figure~\ref{fig:8_10}, where it is shown that $u_q(8_{10})=1$.

\subsection{The Montesinos trick}
One of the most successful tools in studying the classical unknotting number has been the Montesinos trick \cite{Montesinos1975surgery}, which relates crossing changes in a link $L$ to surgeries on its double branched cover $\Sigma(L)$.  In this paper, we derive the appropriate version of the Montesinos trick for proper RTRs and use it to prove some results about knots with proper rational unknotting number one.

\begin{restatable}{lem}{montytrick}\label{lem:monty}
Suppose that a link $L'$ can be obtained by a proper RTR on $L$. Then there is a 3-manifold $M$ with a single torus boundary component with slopes $\alpha, \beta$ on $\partial M$ such that $M(\alpha)\cong \Sigma(L)$, $M(\beta)\cong \Sigma(L')$ and $\Delta(\alpha, \beta)$ is even.
\end{restatable}
Here $\Delta(\alpha, \beta)$ denotes the distance between the slopes $\alpha$ and $\beta$, that is, the minimal possible number of intersection points when $\alpha$ and $\beta$ are represented by simple closed curves. The classical Montesinos trick asserts that when $L$ and $L'$ are related by a crossing change we can find such slopes with $\Delta(\alpha, \beta)=2$.

\begin{figure}[!ht]
  \begin{overpic}[width=300pt]{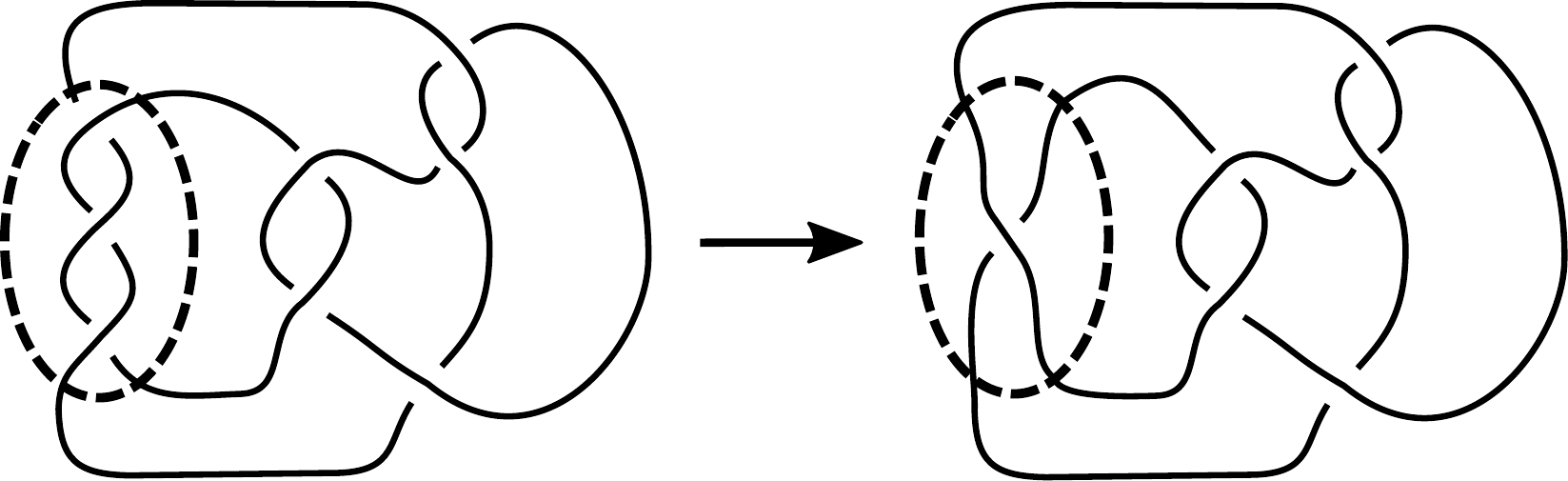}
    \put (15,-5) {\large $8_{10}$}
    \put (80,-5) {\large $U$}
  \end{overpic}
    \vspace{0.7cm}
  \caption{An example of a proper RTR that converts the knot $8_{10}$ to the unknot.}
  \label{fig:8_10}
\end{figure}

When specialized to RTRs on the unknot, Lemma~\ref{lem:monty} implies the following obstruction to a knot having $u_q(K)=1$.
\begin{cor}\label{cor:prun1}
Let $K$ be a non-trivial knot with $u_q(K)=1$, then there is a knot $L$ in $S^3$ and a rational number $p/q \in \Q$ with $q$ even, such that $\Sigma(K)\cong S_{p/q}^3(L)$.
\end{cor}
The notation here is that $S_{p/q}^3(L)$ denotes $p/q$-surgery on the knot $L$. The condition that $\Delta(\alpha, \beta)$ is even in Lemma~\ref{lem:monty} manifests as the condition that $q$ is even, since in the usual conventions for labelling slopes on a knot complement in $S^3$, the quantity $q$ is the distance between the slope $p/q$ and the meridian of the knot.

As a first example, we can combine the work of Ghiggini with Corollary~\ref{cor:prun1} to calculate the proper rational unknotting number of the torus knot $T(3,5)$. 
\begin{exam}
For the torus knot $T(3,5)$, we have $u_q(T_{3,5})=2$.
\end{exam}
\begin{proof}
The double branched cover of $T_{3,5}$ is the Poincar\'e homology sphere. However Ghiggini has shown that the only way the Poincar\'e homology sphere can arise by surgery on a knot in $S^3$, is by $-1$-surgery on the left-handed trefoil \cite{Ghiggini2008fibredness}. Thus Corollary~\ref{cor:prun1} implies that $u_q(T_{3,5})>1$. The knot $T_{3,5}$ is isotopic to the pretzel knot $P(-2,3,5)$ which can be easily unknotted by two proper RTRs (see Lemma~\ref{lem:unknotting_Montesinos}).
\end{proof}

\subsection{Alternating knots}
Combining Corollary~\ref{cor:prun1} with known results about Dehn surgeries on $S^3$ rapidly yields several applications. Firstly we can classify the alternating knots with proper rational unknotting number one. 
\begin{restatable}{thm}{altthm}\label{thm:alt_char}
Let $K$ be a non-trivial alternating knot. Then the following are equivalent:
\begin{enumerate}[(i)]
\item\label{it:uq=1} $u_q(K)=1$;
\item\label{it:surgery} $\Sigma(K) \cong S_{p/q}^3(L)$ where $L\subseteq S^3$ is a knot and $q$ is even;
\item\label{it:tangle_replace} $K$ possesses an alternating diagram which can be obtained by a proper RTR on the dealternating crossing in an almost-alternating diagram of the unknot.
\end{enumerate}
\end{restatable}
We remind the reader that an almost-alternating diagram is one that can made alternating by changing a single crossing, known as the dealternating crossing. Here the implications \eqref{it:uq=1}$\Rightarrow$\eqref{it:surgery} and \eqref{it:tangle_replace}$\Rightarrow$\eqref{it:uq=1} are straight forward with the former being Corollary~\ref{cor:prun1} and the latter being the definition of proper rational unknotting number. The remaining implication, \eqref{it:surgery}$\Rightarrow$\eqref{it:tangle_replace} will be deduced from some previous work of the first author \cite{McCoy2015noninteger}.

\subsection{Prime knots}
We see also that knots with $u_q(K)=1$ are prime. This generalizes the result of Scharlemann \cite{Scharlemann1985unknotting} that knots with unknotting number one are prime. The argument used is essentially the one used by Zhang in his reproof of Scharlemann's result \cite{Zhang1991Unknotting}.
\begin{thm}\label{thm:prime}
Let $K$ be an knot with $u_q(K)=1$. Then $K$ is prime.
\end{thm}
\begin{proof}
If $K$ is a knot with $u_q(K)=1$, then there is a knot $L$ such that $S_{p/q}^3(L)\cong \Sigma(K)$ for some rational $p/q\in \Q$ with $q$ even. In particular, we have $|q|>1$. Thus the work of Gordon and Luecke implies that $S_{p/q}^3(L)\cong \Sigma(K)$ is prime \cite{Gordon1987reducible}. However, since $\Sigma(K)$ is prime, the knot $K$ must itself be prime \cite{Kim1980Splitting}. 
\end{proof}

Note that Theorem~\ref{thm:prime} relies on the fact that we are considering proper RTRs. For example, one can show that the composite knot $3_1\# 4_1$ can be unknotted by a single rational tangle replacement (see Figure~\ref{fig:non_prime}).

\subsection{Montesinos knots}
Finally we study the proper rational unknotting number for Montesinos knots.

\begin{thm}\label{thm:Montesinos}
Let $K$ be a Montesinos knot with containing at least four rational tangles, that is $K=\M(e; \frac{p_1}{q_1}, \dots, \frac{p_k}{q_k})$, where $|p_i|\geq 2$ for all $i$ and $k\geq 4$. Then $u_q(K)>1$.
\end{thm}
\begin{proof}
If $K$ is a Montesinos knot with at least four rational tangles, then the double branched cover is a Seifert fibered space with at least four exceptional fibers. Such a Seifert fibered space contains an incompressible torus and is, therefore, Haken. Boyer and Zhang have shown that a Haken Seifert fibered space can arise only by integer surgery on $S^3$ \cite[Corollary~J]{Boyer1994exceptional}. Thus such a Montesinos knot cannot have $u_q(K)=1$.
\end{proof}
Again, Theorem~\ref{thm:Montesinos} is specific to the proper rational unknotting number; there are Montesinos knots with four rational tangles that can be unknotted by a non-proper RTR, as shown in Figure~\ref{fig:monty}. 

Montesinos knots with three rational tangles are harder to deal with on account of our weaker understanding of atoroidal Seifert fibered surgeries. 

\begin{restatable}{thm}{montythreestrands}\label{thm:Montesinos_3strands}
Let $K=\M(e; \frac{p_1}{q_1}, \frac{p_2}{q_2}, \frac{p_3}{q_3})$ be a Montesinos knot with $|p_1|,|p_2|,|p_3|>1$ and $u(K)\geq 5$. Then the following are equivalent: 
\begin{enumerate}[(i)]
\item\label{it:monty_uq=1} $u_q(K)=1$;
\item\label{it:monty_surgery} $\Sigma(K)$ arises by $p/q$-surgery on a torus knot for some $p/q\in \Q$ with $q$ even;
\item\label{it:monty_list} $K$ can be written in the form $K=\M(0;\frac{a}{b}, \frac{c}{d}, \frac{r}{s})$, where $\frac{b}{a}+\frac{d}{c}=\pm \frac{1}{ac}$ and $s$ is even.
\end{enumerate}
\end{restatable}
This result is sufficient to compute the proper rational unknotting number for many Montesinos knots with three rational tangles, since every Montesinos knot with at most three rational tangle has $u_q(K)\leq 2$ (see Lemma~\ref{lem:unknotting_Montesinos} below).

\begin{rem}
The arithmetic conditions in \eqref{it:monty_list} are such that replacing the rational tangle corresponding to the parameter $r/s$ with a trivial tangle results in a copy of the unknot. Furthermore conditions \eqref{it:monty_surgery} and \eqref{it:monty_list} are exactly related by the Montesinos trick. If one takes the branched cover of the tangle $\M(0;\frac{a}{b}, \frac{c}{d}, \star)$, where $\frac{b}{a}+\frac{d}{c}=\pm \frac{1}{ac}$ and $\star$ indicates that we deleted the interior of a ball in place of a rational tangle, then one obtains the complement of a torus knot in $S^3$ (see, for example, the argument proving \cite[Proposition~1.5]{Donald2019diagrams}). Conversely, if one quotients the complement of a torus knot by a strong involution, then one obtains exactly a tangle of this form (see, for example, \cite[Theorem 6.1 and Figure 6]{Zentner_SU2_simple}).
\end{rem}


The implications \eqref{it:monty_surgery}$\Rightarrow$\eqref{it:monty_list}$\Rightarrow$\eqref{it:monty_uq=1} are all easily established without using the hypothesis that $u(K)\geq 5$. The condition on the unknotting number exists purely to aid in the proof of \eqref{it:monty_uq=1}$\Rightarrow$ \eqref{it:monty_surgery}. In particular, the unknotting number condition guarantees that the surgery coming from Corollary~\ref{cor:prun1} satisfies $q\geq 10$, allowing us to invoke Lackenby and Meyerhoff's work on exceptional surgeries \cite{Lackenby2013exceptional}. However, conjecturally Seifert fibered manifolds cannot arise by non-integer surgeries on hyperbolic knots in $S^3$.
\begin{conj}[{\cite[Conjecture~4.8]{Gordon1998dehn}}]\label{conj:hyperbolic_seifert_fibered_surgeries}
Suppose $K \subseteq S^3$ is a hyperbolic knot and that $S^3_{p/q}(K)$ is a Seifert fibered space. Then $q = 1$. 	
\end{conj}
If one runs through the proof of Theorem~\ref{thm:Montesinos_3strands} using Conjecture~\ref{conj:hyperbolic_seifert_fibered_surgeries} in the place of results from \cite{Lackenby2013exceptional}, then one sees that Conjecture~\ref{conj:hyperbolic_seifert_fibered_surgeries} implies the following conjecture, which characterizes the Montesinos knots with rational unknotting number one.
\begin{conj}\label{conj:3_tangles}
Conditions \eqref{it:monty_uq=1}, \eqref{it:monty_surgery} and \eqref{it:monty_list} from Theorem~\ref{thm:Montesinos_3strands} are equivalent for all Montesinos knots. 
\end{conj}
Note that Theorem~\ref{thm:Montesinos} implies the conjecture for Montesinos knots with at least four rational parameters. In fact, given Theorem~\ref{thm:alt_char} and Theorem~\ref{thm:Montesinos_3strands}, the only cases of Conjecture~\ref{conj:3_tangles} that remain are those of non-alternating Montesinos knots with three rational parameters and $u(K)\leq 4$.

\subsection*{Acknowledgements} The authors would like to thank Lukas Lewark for enlightening correspondence and for sharing an early copy of \cite{Iltgen2021Khovanov} which sparked their interest in the proper rational unknotting number. The second author would also like to thank Luisa Paoluzzi for enlightening correspondence, and he is grateful for support through the Heisenberg program of the DFG. 

\begin{figure}
  \begin{overpic}[width=250pt]{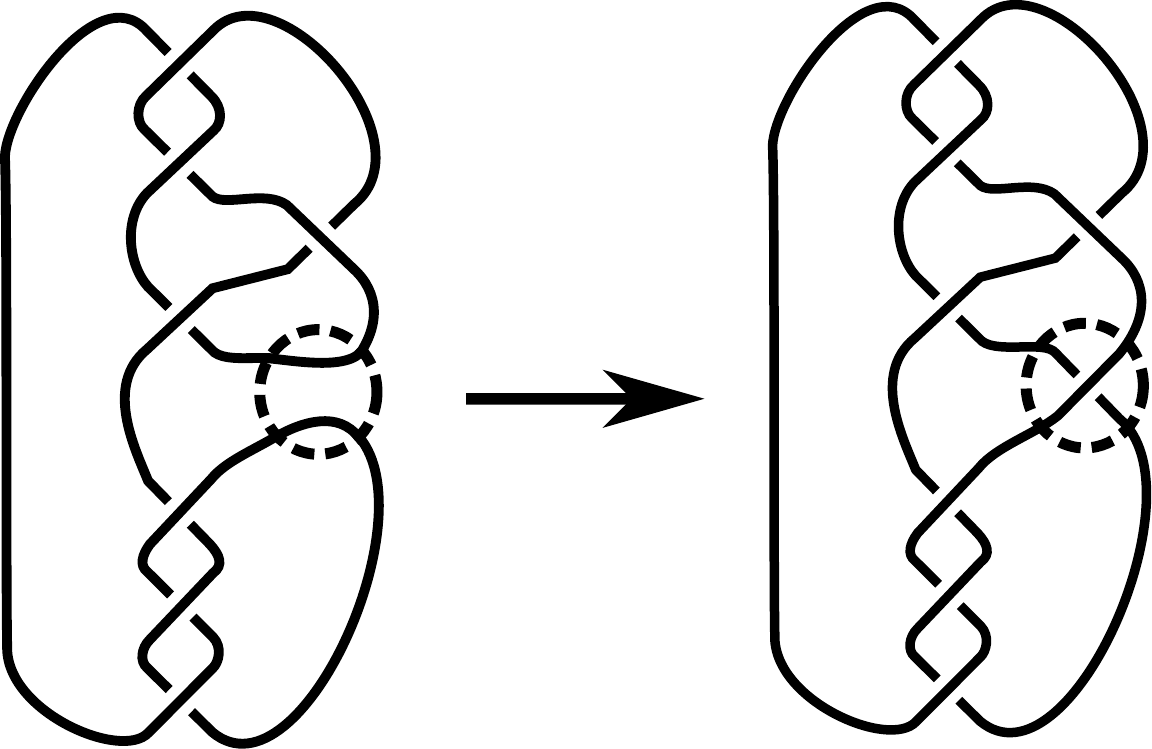}
    \put (12,-5) {\large $3_1 \# 4_1$}
    \put (80,-5) {\large $U$}
  \end{overpic}
    \vspace{0.7cm}
  \caption{The non-prime knot $3_1 \# 4_1$ can be unknotted by a non-proper RTR.}
  \label{fig:non_prime}
\end{figure}

\begin{figure}[!ht]
  \begin{overpic}[width=300pt]{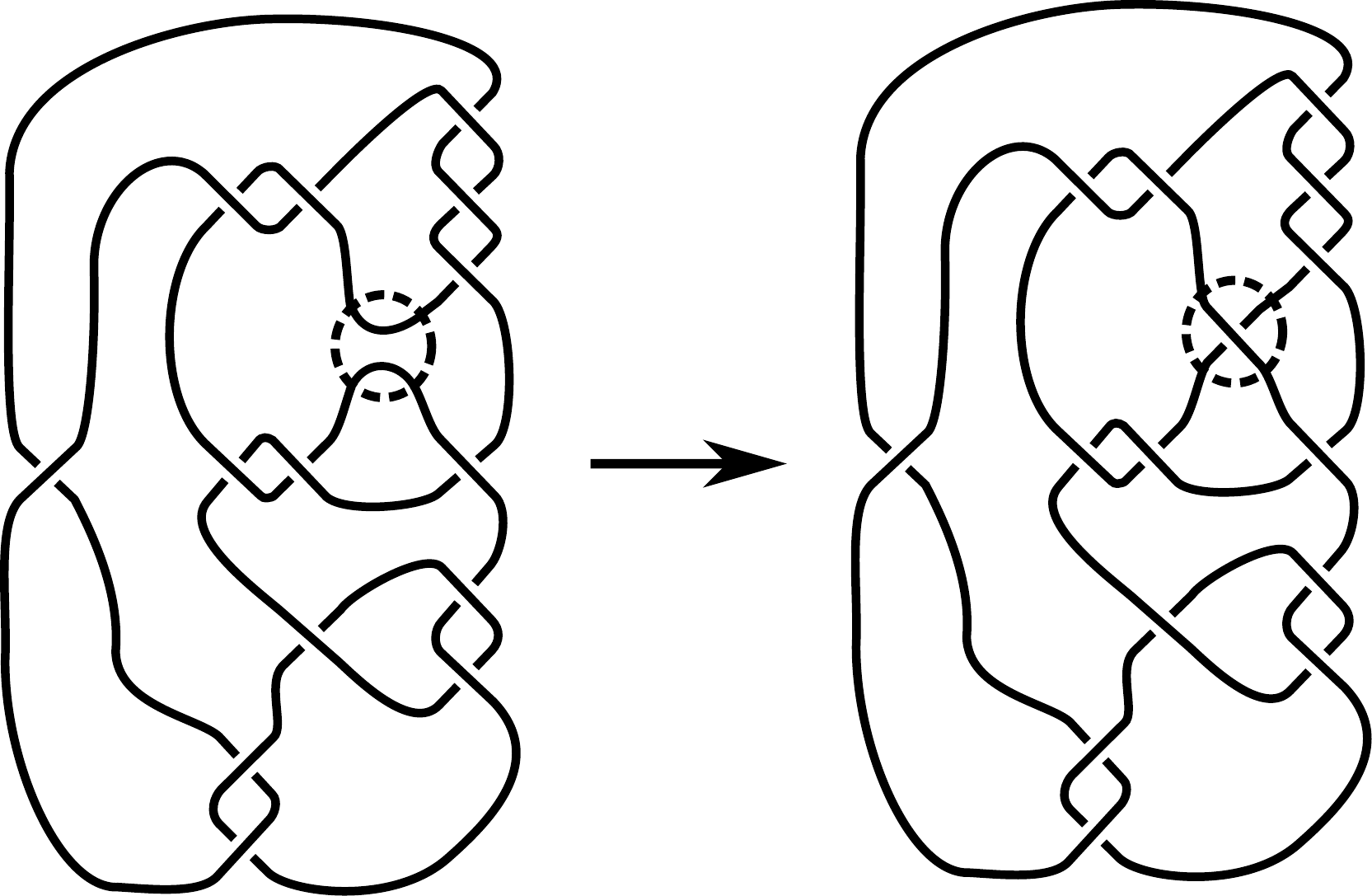}
    \put (80,-5) {\large $U$}
    \put (3,-5) {\large $\M\left(-1;\frac21,\frac32,\frac31,\frac73\right)$}
  \end{overpic}
    \vspace{0.5cm}
  \caption{The Montesinos knot $\M\left(-1;\frac21,\frac32,\frac31,\frac73\right)$ with four rational tangles that can be unknotted by a non-proper RTR.}
  \label{fig:monty}
\end{figure}

\section{Rational tangles}
We begin with a review of rational tangles and their properties. A more detailed treatment of rational tangles can be found in \cite{BurdeZieschang} or \cite{Gordon2009dehnsurgery}. A tangle in $B^3$ is a pair $(B^3, A)$, where $A$ is a properly embedded 1-manifold. We say $(B^3, A)$ is a marked tangle if $\partial B^3 \cap A$ consists of four points, and we have fixed an identification of the pairs $(\partial B^3 , \partial B^3 \cap A)$ and $(S^2,\{ NE,NW,SE,SW\})$ as shown in Figure~\ref{fig:tangle_defs}(a). Two marked tangles are considered equivalent if they are isotopic via an isotopy fixing the boundary. A rational tangle is a marked tangle $(B^3, A)$, which is homeomorphic to the trivial tangle, as shown in Figure~\ref{fig:tangle_defs}(b). For any $p/q\in \Q\cup \{1/0\}$, there is a rational tangle $R(p/q)$ built up from the tangles $R(1/0)$ and $R(0/1)$ using the relationships depicted in Figure~\ref{fig:rational_tangles}. It is a remarkable theorem of Conway that every rational tangle is equivalent as a marked tangle to precisely one of these $R(p/q)$ \cite{Conway1969algebraic}. 

\begin{figure}[!ht]
  \begin{overpic}[width=200pt]{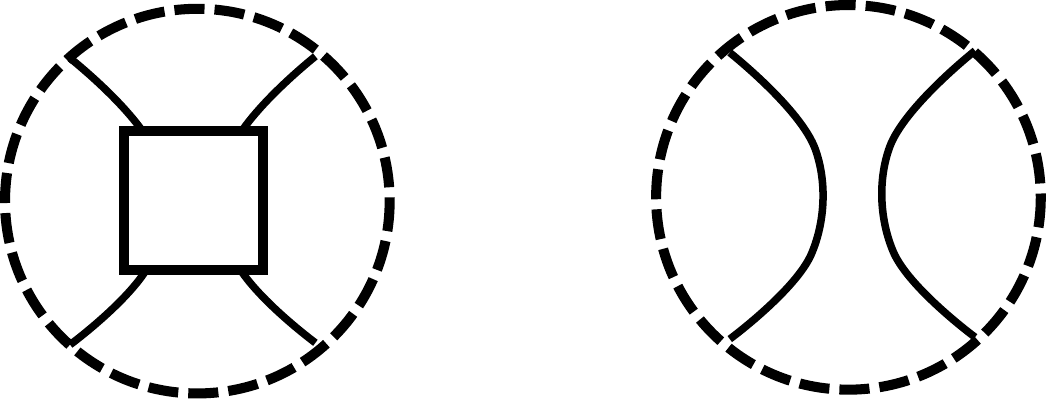}
    \put (17,-8) {(a)}
    \put (80,-8) {(b)}
    \put (-9,32) {NW}
    \put (36,32) {NE}
    \put (36,0) {SE}
    \put (-9,0) {SW}
    \put (17,16) {$T$}
  \end{overpic}
    \vspace{0.7cm}
  \caption{(a) The identifications on the boundary of a marked tangles and (b) the trivial tangle}
  \label{fig:tangle_defs}
\end{figure}

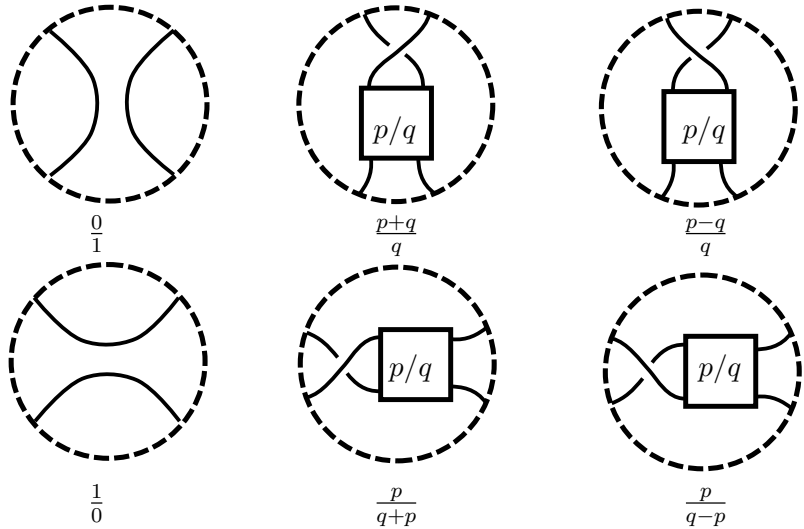
\begin{figure}[!ht]
  \begin{overpic}[width=300pt]{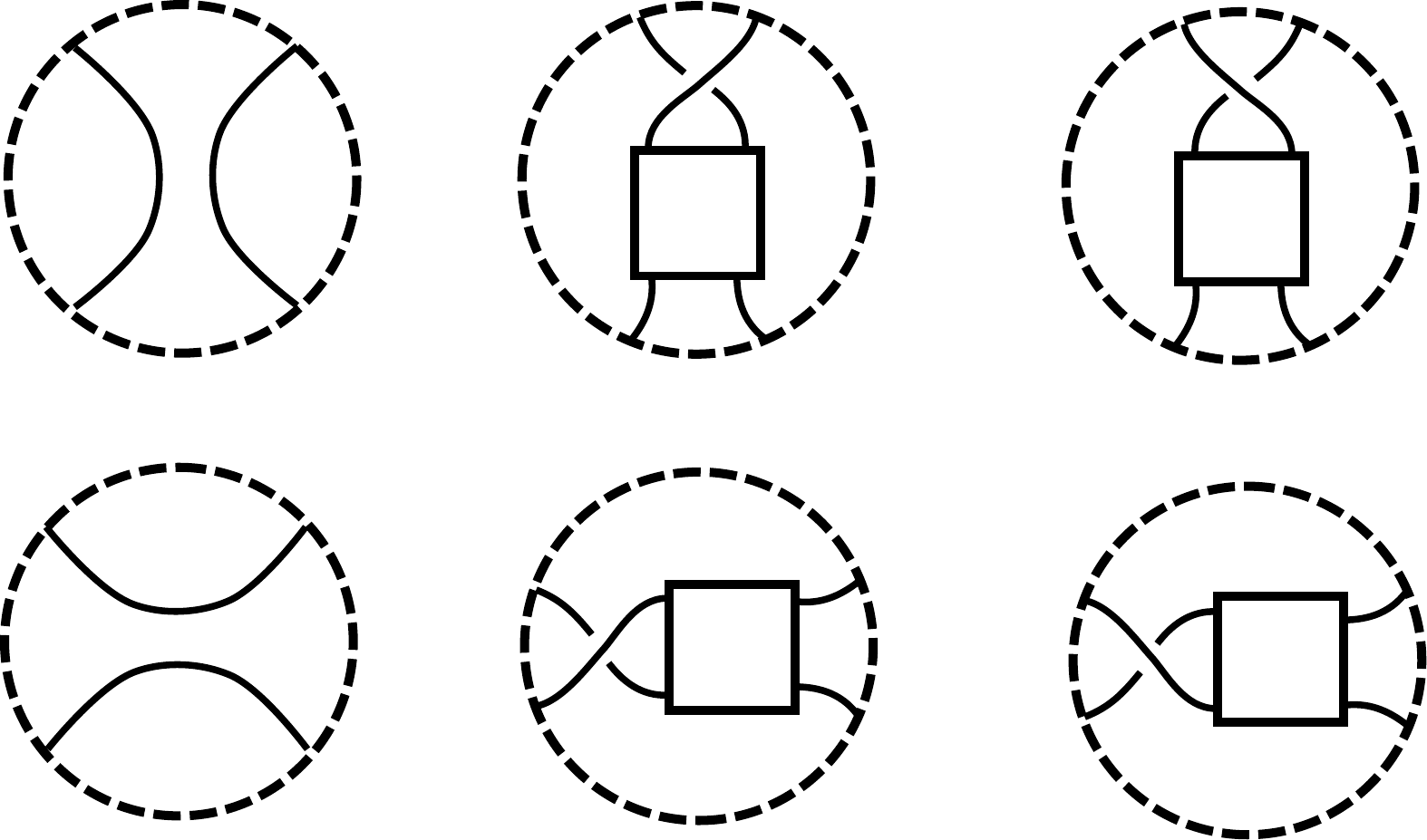}
    \put (10,29) {\large $\frac01$}
    \put (10,-5) {\large$\frac10$}
    \put (46,29) {\large $\frac{p+q}{q}$}
    \put (46,-5) {\large $\frac{p}{q+p}$}
    \put (85,29) {\large $\frac{p-q}{q}$}
    \put (85,-5) {\large $\frac{p}{q-p}$}
    \put (46,42) {$p/q$}
    \put (48,12) {$p/q$}
    \put (85,42) {$p/q$}
    \put (87,12) {$p/q$}
  \end{overpic}
    \vspace{0.7cm}
  \caption{Building up rational tangles.}
  \label{fig:rational_tangles}
\end{figure}
Using the relationships in Figure~\ref{fig:rational_tangles} it is easy to check that the  connectivity of the endpoints of the tangle $R(p/q)$ is determined by the parity of $p$ and $q$ as shown in Figure~\ref{fig:rational_parity}.
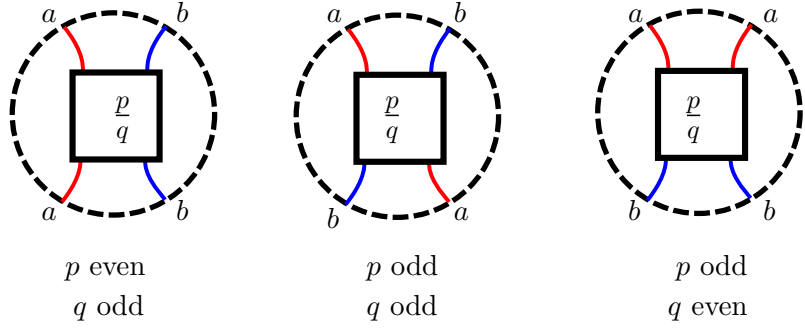
\begin{figure}[!ht]
\vspace{0.5cm}
  \begin{overpic}[width=300pt]{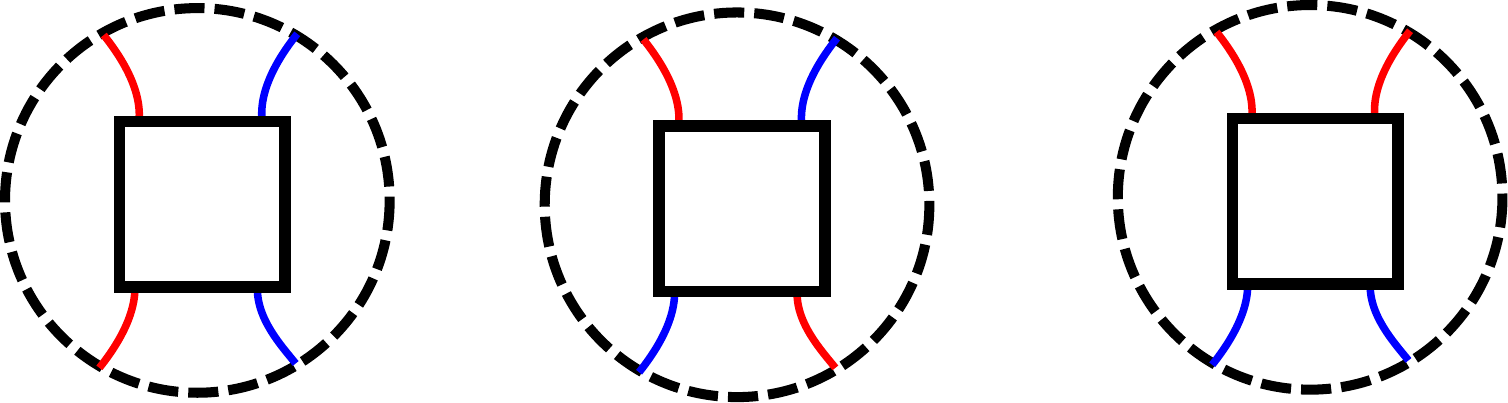}
    \put (13,12) {\Large $\frac{p}{q}$}
    \put (47,12) {\Large $\frac{p}{q}$}
    \put (85,12) {\Large $\frac{p}{q}$}
    \put (7,-7) {$p$ even}
    \put (45,-7) {$p$ odd}
    \put (84,-7) {$p$ odd}
     \put (8,-12) {$q$ odd}
    \put (45,-12) {$q$ odd}
    \put (83,-12) {$q$ even}
    \put (4,0) {$a$}
    \put (4,25) {$a$}
    \put (21,0) {$b$}
    \put (21,25) {$b$}
    \put (40,-1) {$b$}
    \put (40,25) {$a$}
    \put (56,0) {$a$}
    \put (56,25) {$b$}
    \put (78,0) {$b$}
    \put (78,25) {$a$}
    \put (95,0) {$b$}
    \put (95,25) {$a$}
  \end{overpic}
  \vspace{1.3cm}
  \caption{The configuration of endpoints on the boundary of a rational tangle depend on the parities of $p$ and $q$ (the strands of the same color are connected).}
  \label{fig:rational_parity}
\end{figure}

Thus we see that if one performs a RTR that replaces a copy of $R(r/s)$ with $R(p/q)$, this is a proper rational tangle replacement if and only if $r\equiv p \bmod 2$ and $s\equiv q \bmod 2$.
However, one easily finds that this parity condition is equivalent to the condition that $|rq-sp|\equiv 0 \bmod 2$.
\footnote{One can either check this case by case for all possible parities of $p,q,r$ and $s$ or one can use the fact that $rq-sp$ is the determinant of the matrix $\begin{pmatrix}
p & r\\ q & s
\end{pmatrix}$ and consider when this matrix is invertible mod 2.}
We summarize this discussion in the following lemma.
\begin{lem}\label{lem:parity_cond}
A rational tangle replacement that replaces a copy of $R(r/s)$ with a copy of $R(p/q)$ is a proper rational tangle replacement if and only if
\[|rq - ps|\equiv 0 \bmod 2.\]
\end{lem}

Now we recall the general statement of the Montesinos trick that we will use.
\begin{lem}\label{lem:fullmonty}
Suppose that $L'$ is obtained from $L$ by replacing a copy of $R(p/q)$ with a copy of $R(r/s)$. Then there is a 3-manifold $M$ with a single toroidal boundary component and slopes $\alpha$ and $\beta$ on $\partial M$ such that $M(\alpha)\cong \Sigma(L')$, $M(\beta)\cong \Sigma(L)$ and $\Delta(\alpha, \beta)=|ps-rq|$. 
\end{lem}
\begin{proof}
Suppose that the rational tangle replacement takes place in the interior of a 3-ball $B \subseteq S^3$. Since the double branched cover of a rational tangle is homeomorphic to $S^1\times D^2$, we see that both $\Sigma(L)$ and $\Sigma(L')$ are obtained by Dehn filling on $M$, where $M$ is the manifold obtained as the double branched cover of $S^3 \setminus \mathring{B}$. Let $\Sigma(\partial B)$ be the double cover of the sphere $\partial B$	 branched over $L\cap \partial B= L'\cap \partial B$. As explained in \cite[Theorem~4.3]{Gordon2009dehnsurgery}, there is a choice of curves $\widetilde{\mu}$ and $\widetilde{\lambda}$ forming a basis for $H_1(\Sigma(\partial B);\Z)$, such that for any $a/b \in \Q \cup \{1/0\}$ the curve in the boundary of the double branched cover of $R(a/b)$ that bounds a disk is $a\widetilde{\mu} + b \widetilde{\lambda}$. Thus we see that $\Sigma(L)$ and $\Sigma(L')$ are obtained by filling $M$ along the slopes $\alpha = p\widetilde{\mu} + q \widetilde{\lambda}$ and $\beta = r\widetilde{\mu} + s \widetilde{\lambda}$, respectively. Since the curves $\widetilde{\mu}$ and $\widetilde{\lambda}$ intersect in a single point we the see that distance between $\alpha$ and $\beta$ is given by $\Delta(\alpha, \beta)=|ps-rq|$, as required.
\end{proof}
From this we easily deduce Lemma~\ref{lem:monty}.
\montytrick*
\begin{proof}
Lemma~\ref{lem:parity_cond} implies that if the RTR in Lemma~\ref{lem:fullmonty} is a proper RTR, then the distance $\Delta(\alpha, \beta)=|ps-rq|$ is even as required.
\end{proof}

Finally, we estimate how many crossing changes are required to perform a proper RTR.
\begin{lem}\label{lem:unknotting}
If $q$ is even, then the rational tangle $R(p/q)$ can be converted into $R(1/0)$ using at most $|q/2|$ crossing changes.
\end{lem}
\begin{proof}
We prove this by induction on $|q|$. The base case is $q=0$ for which the lemma is vacuously true. Now suppose that $|q|>0$. By reflecting the tangle if necessary, we can assume that $p/q>0$. Now there is an integer $n\geq 0$ such that $p/q$ can be written in the form $\frac{p}{q}=\frac{nq+r}{q}$, where $0<r<p$. Using the arithmetic associated to rational tangles, we see that $R(p/q)$ can be obtained from $R\left(\frac{r}{q-2r}\right)$ by adding two horizontal twists followed by $n$ vertical twists. Thus by changing one of the crossings in the horizontal twists $R(p/q)$ can be converted into $R\left(\frac{p-2rn}{q-2n}\right)$ by a single crossing change. This is illustrated in Figure~\ref{fig:unknotting}. Since $0<r<q$, we have $|q-2r|<q$ and hence we can assume inductively that $R\left(\frac{p-2rn}{q-2r}\right)$ can be converted into $R(1/0)$ using at most $|\frac{q}{2}|-1$ crossing changes. It follows that $R(p/q)$ can be converted into $R(1/0)$ using at most $|\frac{q}{2}|$ crossing changes, as required.
\end{proof}

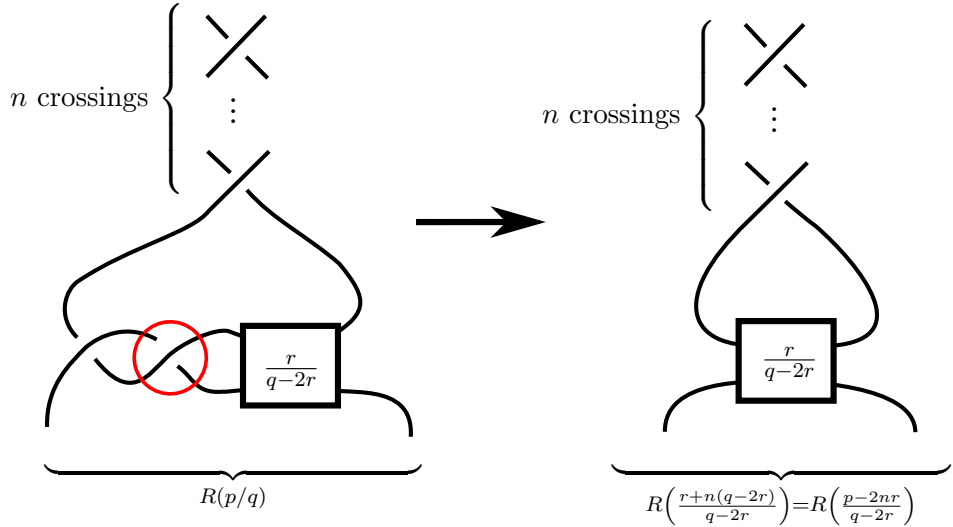
\begin{figure}[!ht]
  \begin{overpic}[width=330pt]{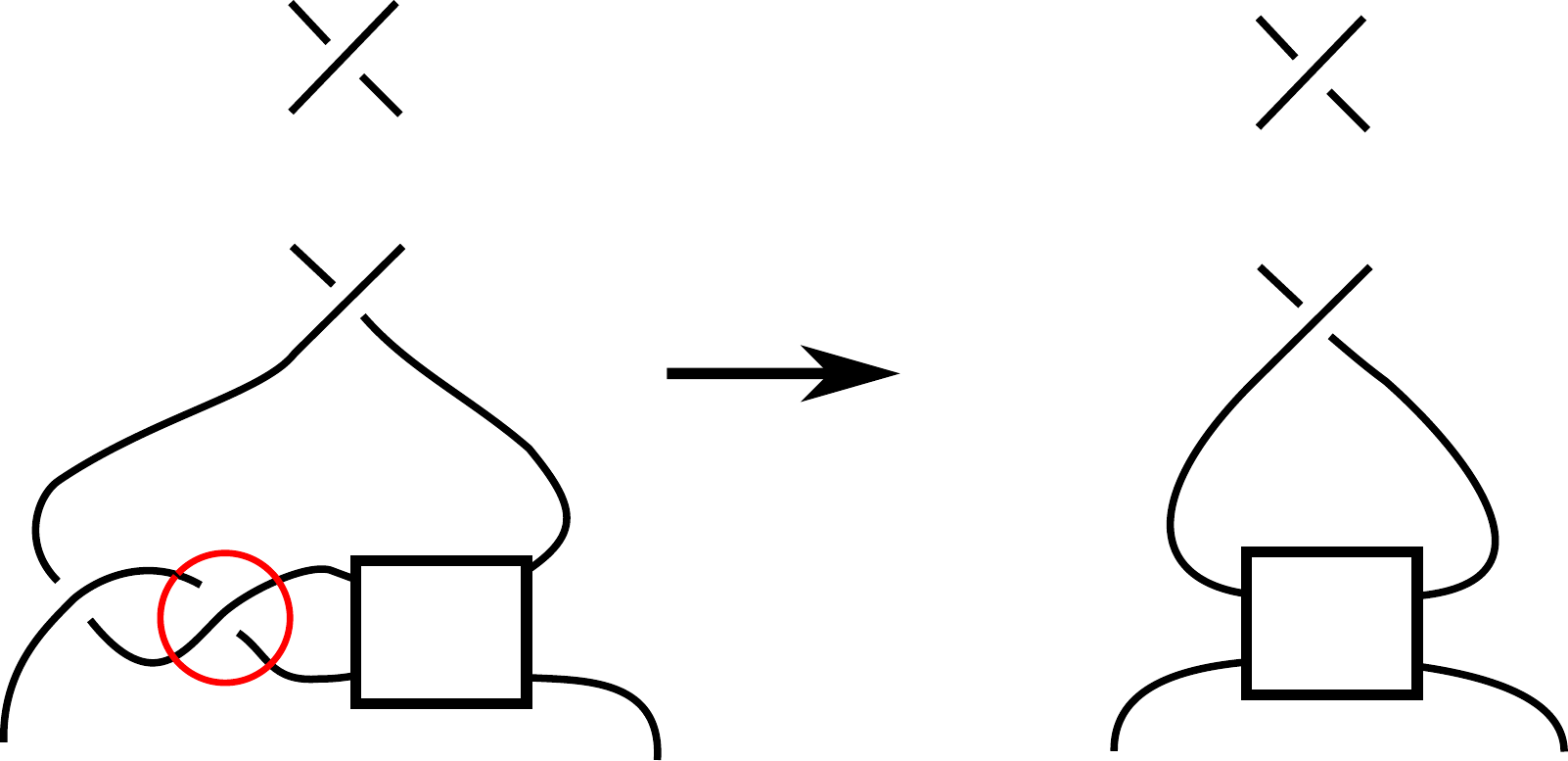}
    \put (25,7) {$\frac{r}{q-2r}$}
    \put (21,36) {\Large $\vdots$}
    \put (83,35) {\Large $\vdots$}
    \put (81,8) { $\frac{r}{q-2r}$}
    \put (-4,38){$n$ crossings $\vbigl\{ \vbigr.$}
    \put (57,36){$n$ crossings $\vbigl\{ \vbigr.$}
    \put (0,-3){$\underbrace{\hspace{5cm}}_{R(p/q)}$}
    \put (65,-3){$\underbrace{\hspace{4.5cm}}_{R\left(\frac{r+n(q-2r)}{q-2r}\right)= R\left(\frac{p-2nr}{q-2r}\right)}$}
  \end{overpic}
  \vspace{1cm}
  \caption{Changing the circled crossing has the effect of transforming $R(p/q)$ into $R(\frac{p-2nr}{q-2r})$.}
  \label{fig:unknotting}
\end{figure}

\section{The proof of Theorem~\ref{thm:alt_char}}

\altthm*
\begin{proof}
The implication \eqref{it:uq=1}$\Rightarrow$\eqref{it:surgery} follows from Corollary~\ref{cor:prun1}. The implication \eqref{it:tangle_replace}$\Rightarrow$\eqref{it:uq=1} follows from the definition of proper rational unknotting number. It remains to establish\eqref{it:surgery}$\Rightarrow$\eqref{it:tangle_replace}. Thus suppose that $K$ is an alternating knot such that $\Sigma(K) \cong S_{p/q}^3(L)$ with $p/q\in \Q$ such that $q$ is even. Note that $p$ is odd since $K$ is a knot.

First we show that we can assume that $|p/q|>1$. Since the double branched cover of an alternating knot is an $L$-space \cite{Ozsvath2005branched}, we have that $|p/q|\geq 2g(L)-1$ \cite{Ozsvath2003Absolutely_graded, Ozsvath2011rationalsurgery}. So if $|p/q|<1$, then $L$ is the unknot. However for surgeries on the unknot we have that $S_{\frac{p}{q+np}}^3(U) \cong S_{\frac{p}{q}}^3(U)$ for any integer $n\in\Z$. Since $p$ is odd and $q$ is even we can find $n\in\Z$ such that $q'=p+np$ is even and satisfying $|q'|<|p|$. Thus by replacing the slope $p/q$ with $p/q'$, we are free to assume that $\Sigma(K) \cong S_{p/q}^3(L)$ for $|p/q|>1$. Furthermore, by reversing orientations if necessary, we can assume that $p/q<-1$. If we write $p/q$ in the form $p/q=-n+r/q$, where $1\leq r <q$, then Theorem~1.3 of \cite{McCoy2015noninteger} implies that $K$ admits an alternating diagram that can be obtained by replacing the dealternating crossing of an almost alternating diagram of the unknot by a copy of $R((q-r)/r)$. In the conventions of \cite[Section~5]{McCoy2015noninteger}, the dealternating crossing is considered as a copy of $R(-1/1)$. Since we are assuming $q$ is even, $r$ is odd. This implies that $q-r$ is also odd and hence this RTR is a proper RTR, as required.
\end{proof}

\section{Montesinos knots}\label{sec:Montesinos_knots}
We briefly establish our conventions on Montesinos knots. We use the notation $\M(e; \frac{p_1}{q_1}, \dots, \frac{p_k}{q_k})$ to depict the knot or link with a diagram as shown in Figure~\ref{fig:Montesinos_diagram}. For an explicit example see the left hand side of Figure~\ref{fig:monty} which depicts a copy of $\M\left(-1;\frac21,\frac32,\frac31,\frac73\right)$. Note that the same link has many many equivalent presentations in this notation. The rational parameters can be cyclically reordered and their overall order can be reversed without changing the isotopy type. Furthermore, we have the relations that
\[
\M\left(e;\frac{p_1}{q_1}, \dots , \frac{p_{k}}{q_{k}}, \frac{1}{0}\right)=\M\left(e;\frac{p_1}{q_1}, \dots , \frac{p_{k}}{q_{k}}\right)
\]
and
\[
\M\left(e;\frac{p_1}{q_1}, \dots , \frac{p_{k}}{q_{k}}\right)=\M\left(e\pm 1;\frac{p_1}{q_1}, \dots , \frac{p_{k-1}}{q_{k-1}}, \frac{p_{k}}{q_{k}\mp p_k}\right).
\]

These conventions are such that the double branched cover of $\M(e; \frac{p_1}{q_1}, \dots, \frac{p_k}{q_k})$ is the Seifert fibered space $S^2(e; \frac{p_1}{q_1}, \ldots, \frac{p_k}{q_k})$, which we take to be the 3-manifold with the surgery presentation given in Figure~\ref{fig:sfs_as_surgery}. A full account of Montesinos links and their double branched covers can be found in \cite[Chapter~12]{BurdeZieschang} for example.

\begin{figure}[!ht]
  \begin{overpic}[width=300pt]{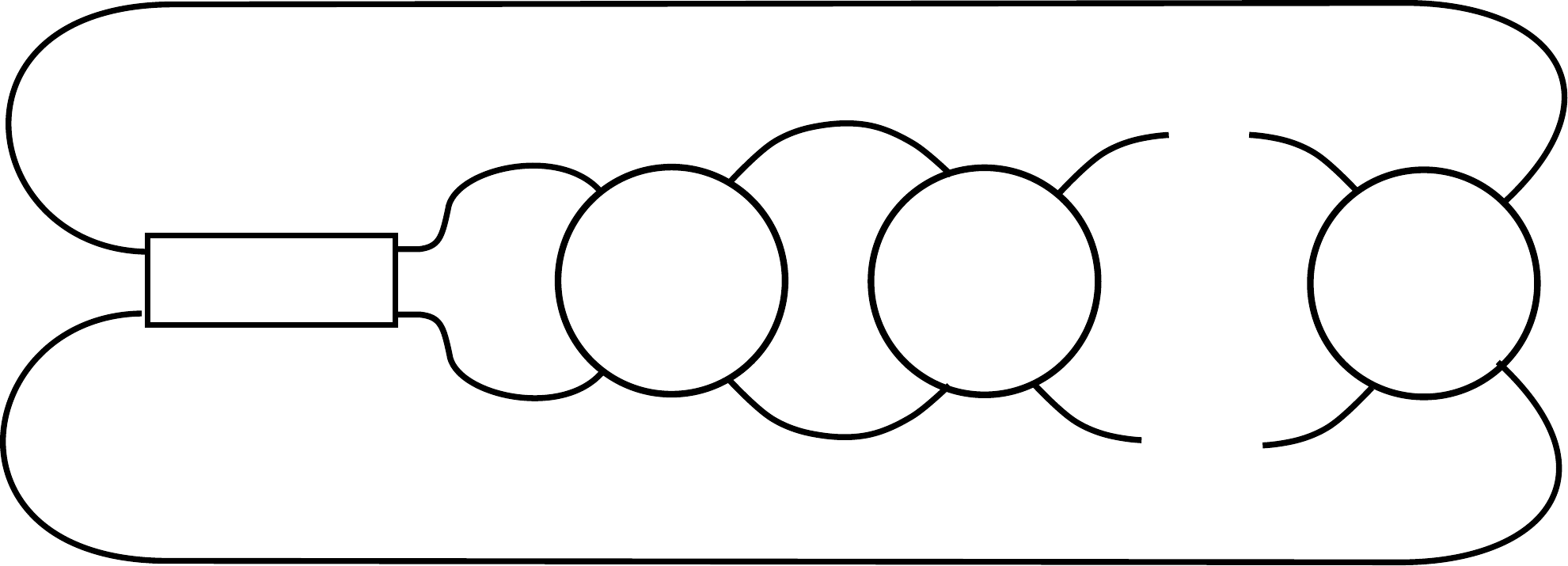}
    \put (41,17) {\large $\frac{p_1}{q_1}$}
    \put (61,17) {\large $\frac{p_2}{q_2}$}
    \put (89,17) {\large $\frac{p_k}{q_k}$}
    \put (14,17) {$e$}
    \put (74,16.5) {\Large $\dots$}
  \end{overpic}
  \caption{The Montesinos link $\M(e; \frac{p_1}{q_1}, \dots, \frac{p_k}{q_k})$.}
  \label{fig:Montesinos_diagram}
\end{figure}

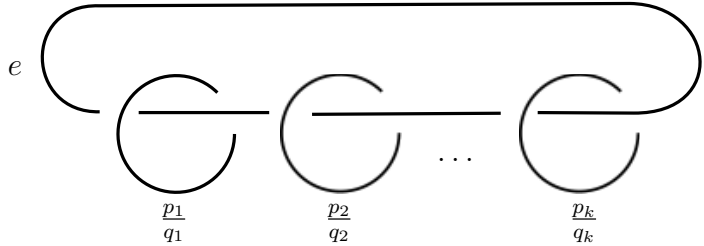
\begin{figure}[!ht]
  \begin{overpic}[width=250pt]{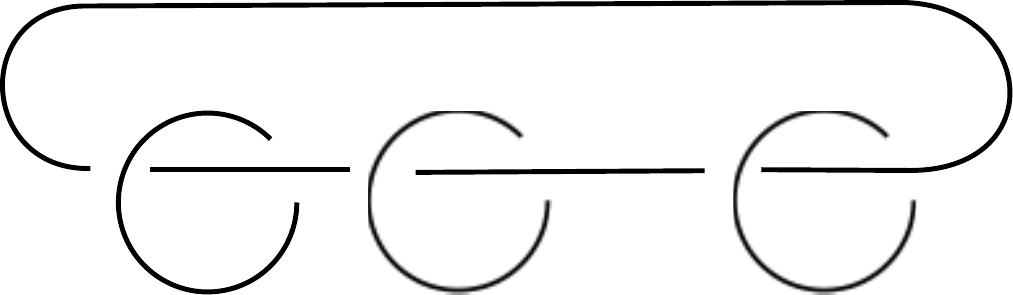}
    \put (-5, 18) {\large $e$}
    \put (60, 5) {\large $\dots$}
    \put (18, -5) {\large $\frac{p_1}{q_1}$}
    \put (43, -5) {\large $\frac{p_2}{q_2}$}
    \put (80, -5) {\large $\frac{p_k}{q_k}$}
  \end{overpic}
  \vspace{0.5cm}
  \caption{Surgery presentation of the Seifert fibered space $S^2(e; \frac{p_1}{q_1}, \ldots, \frac{p_k}{q_k})$.}
  \label{fig:sfs_as_surgery}
\end{figure}

\montythreestrands*
\begin{proof}
We begin with the implication \eqref{it:monty_uq=1} $\Rightarrow$ \eqref{it:monty_surgery}. Let $K$ be such a Montesinos knot with $u_q(K)=1$ and $u(K)\geq 5$. By choosing a suitable diagram we can assume that $K$ can be unknotted by replacing a copy of $R(r/q)$ with a copy of $R(1/0)$, where $q\geq 2$ is even. By Lemma~\ref{lem:unknotting}, we have that $5\leq u(K)\leq q/2$. So Lemma~\ref{lem:fullmonty} implies that we have $\Sigma(K)\cong S_{p/q}^3(L)$ for some knot $L$ in $S^3$ and some $p/q\in \Q$ with $q\geq 10$. Since $K$ has three rational parameters, $\Sigma(K)$ is an atoroidal Seifert fibered space. By Thurston's trichotomy for knots in $S^3$, $L$ is either a torus knot, a hyperbolic knot or a satellite knot. If $L$ is a torus knot, then we already have the desired conclusion. By Lackenby and Meyerhoff's bound on exceptional surgeries  \cite{Lackenby2013exceptional}, the condition $q\geq 10$ implies that $L$ is not a hyperbolic knot. Thus it remains to consider the case that $L$ is a satellite knot. Since $L$ admits a non-integer Seifert fibered surgery, Miyazaki and Motegi have shown that $L$ is cabled exactly once \cite{Miyazaki1997Seifert}. That is, $L$ is a cable of a knot $L'$ where $L'$ is itself a torus knot or a hyperbolic knot. Suppose that $L$ is the cable $C_{\alpha, \beta}\circ L'$. Since $\Sigma(L)$ is atoroidal, the slope $p/q$ must take the form $p/q= \alpha \beta \pm 1/q$. For such a slope we have $S_{p/q}^3(L)\cong S_{p/(q\alpha^2)}^3(L')$ (see for example, \cite[Lemma~3.3]{Gordon1983Satellite}). Since $q\alpha^2\geq 10$, $L'$ cannot be a hyperbolic knot. So we conclude that $\Sigma(K)\cong S_{p/(q\alpha^2)}^3(L')$ for a $L'$ a torus knot. Thus we have established \eqref{it:monty_surgery} even when $L$ is a satellite knot.

Next, the implication \eqref{it:monty_surgery} $\Rightarrow$ \eqref{it:monty_list}. This follows from computing the Seifert fibered surgeries on torus knots. If we perform $p/q$-surgery on the positive torus knot $T_{a,c}$, where $a,c>1$ we obtain the Seifert fibered space $S^2(0;\frac{a}{b}, \frac{c}{d}, \frac{p}{q}-ac)$, where $\frac{b}{a}+\frac{d}{c}=\frac{1}{ac}$. 

This calculation is performed in detail in \cite[Lemma~4.4]{Owens2012negdef} (see also \cite{Moser1971elementary}). Likewise, performing $p/q$-surgery on the negative torus knot $T_{-a,c}$, we obtain a space of the form $S^2(0;-\frac{a}{b}, -\frac{c}{d}, \frac{p}{q}+ac)$. So in either case, $p/q$-surgery on a torus knot with $q$ even yields a Seifert fibered space of the form $S^2(0;\frac{a}{b}, \frac{c}{d}, \frac{r}{s})$, where $\frac{b}{a}+\frac{d}{c}=\pm \frac{1}{ac}$ and $s$ is even. This implies that if $\Sigma(K)$ arises by such a surgery, then $K$ can be written in the required form. 

Finally, the implication \eqref{it:monty_list} $\Rightarrow$ \eqref{it:monty_uq=1}.
Given $K=\M(0;\frac{a}{b}, \frac{c}{d}, \frac{r}{s})$ with $\frac{b}{a}+\frac{d}{c}=\pm \frac{1}{ac}$ and $s$ even, we can perform a proper RTR to replace the $\frac{r}{s}$ tangle with a copy of the trivial tangle $R(1/0)$. This converts $K$ into a 2-bridge knot, which is easily seen to be the unknot, since $\frac{b}{a}+\frac{d}{c}=\pm \frac{1}{ac}$.
\end{proof}

We conclude with an easy upper bound on the proper rational unknotting numbers of Montesinos knots.

\begin{lem}\label{lem:unknotting_Montesinos}
For any Montesinos knot $K=\M(e;\frac{p_1}{q_1}, \dots , \frac{p_r}{q_r})$, with $r\geq 2$ we have $u_q(K)\leq r-1$.
\end{lem}

\begin{proof}
We prove this by induction. If $r\leq 2$, then $K$ is a 2-bridge knot. Since every 2-bridge knot has proper rational unknotting number one, this suffices for the base case. Now suppose that $r\geq 3$. Since $K$ is a knot, at most one of the $p_i$ can be even. Thus, we can assume that $p_r$ is an odd integer. So depending on whether $q_r$ is odd or even we can perform a proper RTR to replace the tangle $R(p_r/q_r)$ in $K$ with a copy of $R(1/1)$ or $R(1/0)$. That is we can transform $K$ into 
\[K'=\M\left(e;\frac{p_1}{q_1}, \dots , \frac{p_{r-1}}{q_{r-1}}, \frac{1}{0}\right)=\M\left(e;\frac{p_1}{q_1}, \dots , \frac{p_{r-1}}{q_{r-1}}\right)\]
or
\[K'=\M\left(e;\frac{p_1}{q_1}, \dots , \frac{p_{r-1}}{q_{r-1}}, \frac{1}{1}\right)=\M\left(e-1;\frac{p_1}{q_1}, \dots , \frac{p_{r-1}}{q_{r-1}}\right),
\]
by a single proper RTR. This is obviously sufficient for the induction step.
\end{proof}
 \bibliographystyle{alpha}
\bibliography{master}

\begin{thebibliography}{DMV19}

\bibitem[BZ94]{Boyer1994exceptional}
S.~Boyer and X.~Zhang.
\newblock Exceptional surgery on knots.
\newblock {\em Bull. Amer. Math. Soc. (N.S.)}, 31(2):197--203, 1994.

\bibitem[BZ03]{BurdeZieschang}
Gerhard Burde and Heiner Zieschang.
\newblock {\em Knots}, volume~5 of {\em de Gruyter Studies in Mathematics}.
\newblock Walter de Gruyter \& Co., Berlin, second edition, 2003.

\bibitem[Con70]{Conway1969algebraic}
J.~H. Conway.
\newblock An enumeration of knots and links, and some of their algebraic
  properties.
\newblock In {\em Computational {P}roblems in {A}bstract {A}lgebra ({P}roc.
  {C}onf., {O}xford, 1967)}, pages 329--358. Pergamon, Oxford, 1970.

\bibitem[DMV19]{Donald2019diagrams}
Andrew Donald, Duncan McCoy, and Faramarz Vafaee.
\newblock On {L}-space knots obtained from unknotting arcs in alternating
  diagrams.
\newblock {\em New York J. Math.}, 25:518--540, 2019.

\bibitem[Ghi08]{Ghiggini2008fibredness}
Paolo Ghiggini.
\newblock Knot {F}loer homology detects genus-one fibred knots.
\newblock {\em Amer. J. Math.}, 130(5):1151--1169, 2008.

\bibitem[GL87]{Gordon1987reducible}
C.~McA. Gordon and J.~Luecke.
\newblock Only integral {D}ehn surgeries can yield reducible manifolds.
\newblock {\em Math. Proc. Cambridge Philos. Soc.}, 102(1):97--101, 1987.

\bibitem[Gor83]{Gordon1983Satellite}
C.~McA. Gordon.
\newblock Dehn surgery and satellite knots.
\newblock {\em Trans. Amer. Math. Soc.}, 275(2):687--708, 1983.

\bibitem[Gor98]{Gordon1998dehn}
Cameron~McA. Gordon.
\newblock Dehn filling: a survey.
\newblock {\em Banach Center Publications}, 42:129--144, 1998.

\bibitem[Gor09]{Gordon2009dehnsurgery}
Cameron Gordon.
\newblock Dehn surgery and 3-manifolds.
\newblock In {\em Low dimensional topology}, volume~15 of {\em IAS/Park City
  Math. Ser.}, pages 21--71. Amer. Math. Soc., Providence, RI, 2009.

\bibitem[ILM21]{Iltgen2021Khovanov}
Damian Iltgen, Lukas Lewark, and Laura Marino.
\newblock Khovanov homology and rational unknotting.
\newblock preprint, 2021.

\bibitem[KT80]{Kim1980Splitting}
Paik~Kee Kim and Jeffrey~L. Tollefson.
\newblock Splitting the {PL} involutions of nonprime {$3$}-manifolds.
\newblock {\em Michigan Math. J.}, 27(3):259--274, 1980.

\bibitem[LM13]{Lackenby2013exceptional}
Marc Lackenby and Robert Meyerhoff.
\newblock The maximal number of exceptional {D}ehn surgeries.
\newblock {\em Invent. Math.}, 191(2):341--382, 2013.

\bibitem[McC15]{McCoy2015noninteger}
Duncan McCoy.
\newblock Non-integer surgery and branched double covers of alternating knots.
\newblock {\em J. Lond. Math. Soc. (2)}, 92(2):311--337, 2015.

\bibitem[MM97]{Miyazaki1997Seifert}
Katura Miyazaki and Kimihiko Motegi.
\newblock Seifert fibred manifolds and {D}ehn surgery.
\newblock {\em Topology}, 36(2):579--603, 1997.

\bibitem[Mon75]{Montesinos1975surgery}
Jos\'{e}~M Montesinos.
\newblock Surgery on links and double branched covers of ${S}^3$.
\newblock {\em Knots, groups, and 3-manifolds (Papers dedicated to the memory
  of R. H. Fox)}, pages 227--259, 1975.

\bibitem[Mos71]{Moser1971elementary}
Louise Moser.
\newblock Elementary surgery along a torus knot.
\newblock {\em Pacific J. Math.}, 38:737--745, 1971.

\bibitem[OS03]{Ozsvath2003Absolutely_graded}
Peter Ozsv{\'a}th and Zolt{\'a}n Szab{\'o}.
\newblock Absolutely graded {F}loer homologies and intersection forms for
  four-manifolds with boundary.
\newblock {\em Adv. Math.}, 173(2):179--261, 2003.

\bibitem[OS05]{Ozsvath2005branched}
Peter Ozsv{\'a}th and Zolt{\'a}n Szab{\'o}.
\newblock On the {H}eegaard {F}loer homology of branched double-covers.
\newblock {\em Adv. Math.}, 194(1):1--33, 2005.

\bibitem[OS11]{Ozsvath2011rationalsurgery}
Peter Ozsv{\'a}th and Zolt{\'a}n Szab{\'o}.
\newblock Knot {F}loer homology and rational surgeries.
\newblock {\em Algebr. Geom. Topol.}, 11(1):1--68, 2011.

\bibitem[OS12]{Owens2012negdef}
Brendan Owens and Sa{\v{s}}o Strle.
\newblock Dehn surgeries and negative-definite four-manifolds.
\newblock {\em Selecta Math. (N.S.)}, 18(4):839--854, 2012.

\bibitem[Sch85]{Scharlemann1985unknotting}
Martin~G. Scharlemann.
\newblock Unknotting number one knots are prime.
\newblock {\em Invent. Math.}, 82(1):37--55, 1985.

\bibitem[Zen17]{Zentner_SU2_simple}
Raphael Zentner.
\newblock A class of knots with simple {$SU(2)$}-representations.
\newblock {\em Selecta Math. (N.S.)}, 23(3):2219--2242, 2017.

\bibitem[Zha91]{Zhang1991Unknotting}
Xingru Zhang.
\newblock Unknotting number one knots are prime: a new proof.
\newblock {\em Proc. Amer. Math. Soc.}, 113(2):611--612, 1991.

\end{thebibliography}
\end{document}